\documentclass{amsart}
\usepackage{titletoc}
\usepackage{blindtext}
\usepackage{amsfonts,amssymb,amsmath,amsthm}
\usepackage{enumerate}

\newtheorem{innercustomgeneric}{\customgenericname}
\providecommand{\customgenericname}{}
\newcommand{\newcustomtheorem}[2]{%
  \newenvironment{#1}[1]
  {%
   \renewcommand\customgenericname{#2}%
   \renewcommand\theinnercustomgeneric{##1}%
   \innercustomgeneric
  }
  {\endinnercustomgeneric}
}

\newcustomtheorem{customthm}{Theorem}
\newcustomtheorem{customlemma}{Lemma}

\usepackage{graphicx, color}

\makeatletter
\@namedef{subjclassname@2020}{\textup{2020} Mathematics Subject Classification}
\makeatother

\makeatletter

\newtheorem*{thrm}{Theorem}
\newtheorem{lem}{Lemma}
\newtheorem{prop}{Proposition}
\newtheorem{cor}{Corollary}

\newtheorem{definition}{Definition}
\newtheorem{question}{Question}
\newtheorem{remark}{Remark}
\newtheorem{ex}{Example}
\numberwithin{equation}{section}

\begin{document}
\title[Rescaled topological entropy]{Rescaled topological entropy}

\author{E. Rego}
\address{AGH University of Science and Technology
Krakow, Poland.}
\email{rego@agh.edu.pl}

\author[C. Rojas]{C. Rojas}
\address{Hangzhou International Innovation Institute of Beihang University,  
Hangzhou 311115, China.}
\email{tchivatze@gmail.com}

\author[X. Wen]{X. Wen}
\address{LMIB, Institute of Artificial Intelligence and School of Mathematical Science, Beihang University, Beijing, China.}
\email{wenxiao@buaa.edu.cn}


\begin{abstract}
We prove that to any smooth vector field of a closed manifold it can be assigned a nonnegative number called {\em rescaled topological entropy} satisfying the following properties:
it is an upper bound for both the topological entropy and the rescaled metric entropy \cite{ww};
coincides with the topological entropy for nonsingular vector fields;
is positive for certain surface vector fields (in contrast to the topological entropy);
is invariant under rescaled topological conjugacy;
and serves as an upper bound for the growth rate of periodic orbits for rescaling expansive flows with dynamically isolated singular set.
Therefore, the rescaled topological entropy bounds such growth rates for $C^r$-generic rescaling (or $k^*$) expansive vector fields on closed manifolds.
\end{abstract}

\maketitle



\section{Introduction}
\label{sec1}

\noindent
Every \( C^1 \) vector field $X$ on a closed surface has zero topological entropy $e(X)$ (see \cite{lg}). Still, such flows can display chaotic behavior \cite{u}.
This motivates the following question:

\begin{question}
\label{q1}
Is there any kind of "entropy" allowing to detect chaos for surface flows?
\end{question}

On the other hand, the topological entropy for expansive vector fields bounds the growth rate of periodic orbits from above \cite{bw}. More precisely,
\begin{equation}
\label{precatorio}
\limsup_{t\to\infty}\frac{1}t\log v(t)\leq e(X).
\end{equation}
where $v(t)$ is the number of periodic orbits of period less than or equal to $t$.
This motivates one more question:

\begin{question}
\label{q2}
Is \eqref{precatorio} true for "nearby expansive" vector fields?
\end{question}

To explore such questions we assign to any $C^1$ vector field a nonnegative number called {\em rescaled topological entropy}. We prove that this number satisfies the following properties: it is an upper bound for both the topological entropy and the rescaled metric entropy \cite{ww}; it coincides with the topological entropy for nonsingular vector fields; it is positive for certain surface vector fields; it is invariant under rescaled topological conjugacy; and it serves as an upper bound for the growth rate of periodic orbits for rescaling expansive flows with dynamically isolated singular set. In particular, the rescaled topological entropy bounds such growth rates for $C^r$-generic rescaling (or $k^*$) expansive vector fields on closed manifolds.
Let us state our result in a precise way.

Consider a {\em closed manifold} i.e. a compact connected boundaryless manifold of positive dimension equipped with a Riemannian metric.
Denote by $\|\cdot\|$ and $d$ the norm and the distance of the tangent bundle $TM$ and $M$ generated by the Riemannian metric respectively.

Let $X$ be a $C^1$ vector field on $M$.
Throughout, we assume \(X \neq 0\).
The {\em topological entropy} of \(X\) is defined by
\begin{equation}
\label{eqq1}
e(X)=h(\varphi_1),
\end{equation}
where \(\varphi_t\) is the flow generated by \(X\), \(\varphi_1\) is the time-one map, and \(h(T)\) is the classical topological entropy of a continuous transformation \(T:M \to M\) (as defined in \cite{akm}).

We say that $X$ is {\em rescaling expansive} \cite{wew} if for every $\epsilon>0$ there is $\delta>0$ such that if $x,y\in M$ satisfy
$d(\varphi_s(x),\varphi_{h(s)}(y))\leq\delta\|X(\varphi_s(x))\|$ for all $s\in\mathbb{R}$ and some increasing homeomorphism $h:\mathbb{R}\to\mathbb{R}$, then
$\varphi_{h(0)}(y)=\varphi_{s^*}(x)$ for some $s^*\in [-\epsilon,\epsilon]$.
We say that $x\in M$ is a singularity if $X(x)=0$.
Denote by $Sing(X)$ the set of singularities of $X$.
We say that $X$ is {\em nonsingular} if $Sing(X)=\emptyset$.
$x\in M$ is a {\em periodic point} if there is a minimal $t>0$ (called period) such that $\varphi_t(x)=x$.
A {\em periodic orbit} is the orbit $\{\phi_t(x)\mid t\in\mathbb{R}\}$ of a periodic point $x$, and its period is that of $x$.
Let $v(t)$ denote the number of periodic orbits of period $\leq t$.

Let $N$ be another closed manifold equipped with a $C^1$ vector field $Y$.
A map $h:N\to M$ is {\em rescaled continuous} if for every $\epsilon>0$ there is $\delta>0$ such that
if $y,y'\in N$ and $d(y,y')\leq \delta\|Y(y)\|$, then $d(h(y),h(y'))\leq \epsilon\|X(h(y))\|$.
We say that $h$ is a {\em rescaled homeomorphism} if it is bijective and both $h$ and its inverse $h^{-1}$ are rescaled continuous.
We say that $X$ and $Y$ are {\em rescaled topologically conjugate} if there there is a rescaled homeomorphism $h:N\to M$ satisfying
$\phi_t\circ h=h\circ \varphi_t$ for all $t\in\mathbb{R}$, where $\varphi$ is the flow of $Y$.

Let $\mu$ be a Borel probability measure of $M$ i.e. a $\sigma$-additive measure defined in the Borel $\sigma$-algebra of $M$ with $\mu(M)=1$.
Given $t\geq0$ and $\epsilon,\delta>0$ we
denote by $N^*_\mu(t,\epsilon,\delta)$ the minimal number of rescaled dynamical $(t,\epsilon)$-balls
needed to cover a subset of $\mu$-measure greater than $1-\delta$.
If $\mu(Sing(X))=0$, we define the {\em rescaled metric entropy} \cite{ww}
by
$$
e^*_\mu(X)=\lim_{\delta\to0}\lim_{\epsilon\to0}\limsup_{t\to\infty}\frac{1}t\log N_\mu^*(t,\epsilon,\delta).
$$



With these definitions we can state our result.

\begin{thrm}
To any $C^1$ vector field $X$ of a closed manifold it can be assigned a nonnegative real number $e^*(X)$ (called {\em rescaled topological entropy}) satisfying the following properties:
\begin{enumerate}
\item
If $\mu$ is a Borel probability measure with $\mu(Sing(X))=0$, then
$e^*_\mu(X)\leq e^*(X)$.
\item
\( e(X) \leq e^*(X)\).
\item
\( e(X) = e^*(X) \) if \( X \) is nonsingular.  
\item
There exists a \( C^\infty \) vector field on the two-torus for which \( e^*(X) > 0 \). In particular, the identity \( e(X) =e^*(X) \) is not true in general.
\item
If $X$ is rescaled topologically conjugate to another $C^1$ vector field $Y$, then $e^*(X)=e^*(Y)$. 
\item
If $X$ is rescaling expansive and $Sing(X)$ is dynamically isolated, then
\begin{equation}
\label{precatorinho}
\limsup_{t\to\infty}\frac{1}t\log v(t)\leq e^*(X).
\end{equation}
\end{enumerate}
\end{thrm}

Item (1) is a short of half-variational principle for the rescaled entropy.
Items (3) and (4) say that the rescaled topological entropy is a genuine extension of the topological entropy.
Item (5) says that the rescaled topological entropy is an invariant under rescaled topological conjugacy.
Item (6) gives a positive answer for Question \ref{q2} for rescaled expansive vector fields with dynamically isolated singular set but replacing the topological entropy by the rescaled topological one.

We state two corollaries based on the following definition.
We say that $X$ is {\em k*-expansive} \cite{k} if for every $\epsilon>0$ there is $\delta$ such that if $x,y\in M$ satisfy
$d(\phi_s(x),\phi_{h(s)}(y))\leq \delta$ for every $s\in\mathbb{R}$ and some increasing homeomorphism $h:\mathbb{R}\to\mathbb{R}$ fixing $0$, then $\phi_{h(s_0)}(y)\in \phi_{[s_0-\epsilon,s_0+\epsilon]}(x)$.

Since every k*-expansive vector field is rescaling expansive \cite{rwy}, we have the following corollary.

\begin{cor}
Let $X$ be a $C^1$ vector field of a closed manifold.
If $X$ is k*-expansive and $Sing(X)$ is dynamically isolated, then
\eqref{precatorinho} holds.
\end{cor}

Now, recall that the space of \( C^r \) vector fields on a closed manifold \( M \), for \( r \geq 1 \), is equipped with the \( C^r \)-topology \cite{pdm}.  
A property is said to hold for {\em \( C^r \) generic vector fields on \( M \)} if there exists a residual subset—that is, a countable intersection of open and dense sets—of \( C^r \) vector fields, all of which satisfy the property.
Also recall that a singularity \( x \) of a \( C^r \) vector field \( X \) is hyperbolic if the linear operator \( DX(x) \) has no eigenvalues on the imaginary axis.

It follows from the Kupka-Smale theorem \cite{pdm} that all singularities of a \( C^r \) generic vector field are hyperbolic. Since the hyperbolicity of all singularities implies that the set of singularities is dynamically isolated, we obtain the following corollary.

\begin{cor}
A $C^r$ generic rescaling (or k*) expansive vector field $X$ of a closed manifold satisfies \eqref{precatorinho}.
\end{cor}

Note that \eqref{precatorinho} reduces to \eqref{precatorio} for nonsingular vector fields including the $N$-expansive or CW-expansive vector fields
\cite{acop}, \cite{acp}.
It seems that \eqref{precatorio} holds for the former vector fields but we don't know if this is true for the latter ones.

The remainder of the paper is divided as follows.
In Section \ref{hit}, we define $e^*(X)$ and give some equivalent definitions.
In Section \ref{hop}, we study the rescaled metric entropy.
In Section \ref{hup}, we prove the theorem.

\section{Definition and equivalences}
\label{hit}

\noindent
In this section, we define the number $e^*(X)$ required in the theorem.

\subsection{Bowen-Dinaburg formula}
It follows from \cite{b,d} that the topological entropy of a $C^1$ vector field $X$ of a closed manifold $M$ is given by
\begin{equation}
\label{bolso}
e(X)=\lim_{\epsilon \to 0} \limsup_{t \to \infty} \frac{1}{t}\log R(t,\epsilon),
\end{equation}
where \(R(t,\epsilon)\) is the minimal cardinality of a \emph{\((t,\epsilon)\)-spanning set}, i.e., a set \(F \subset M\) such that
\[
M \;=\; B(F,t,\epsilon)
\quad\text{with}\quad
B(F,t,\epsilon)\;=\;\bigcup_{x \in F} B(x,t,\epsilon).
\]
Here
\begin{equation}
\label{kubrick}
B(x,t,\epsilon)
\;=\;
\{\,y \in M \,:\, d_t(x,y) < \epsilon \}
\quad\quad(x \in M)
\end{equation}
is the {\em dynamical \((t,\epsilon)\)-ball}
induced by the one-parameter family of metrics
\[
d_t(x,y)\;=\;\sup_{0 \le s \le t} d\bigl(\varphi_s(x), \varphi_s(y)\bigr)
\quad(t \ge 0,\, x,y\in M).
\]

\subsection{Definition}
Set $M^*=M\setminus Sing(X)$. Replace \(d_t\) and the dynamical $(t,\epsilon)$-balls by the family of \emph{rescaled} distances
\[
d^*_t(x,y)
\;=\;
\sup_{0 \le s \le t} \frac{d\bigl(\varphi_s(x),\,\varphi_s(y)\bigr)}{\|X(\varphi_s(x))\|}
\quad\quad (t \ge 0,\, x\in M^*,\, y\in M)
\]
and the \emph{rescaled dynamical \((t,\epsilon)\)-balls}
\[
B^*(x,t,\epsilon)
\;=\;
\{\,y \in M : d^*_t(x,y) < \epsilon\}
\quad\quad(\forall\, x \in M^*)
\]
respectively.

If $\delta>0$ we write $M_\delta=\{x\in M:\|X(x)\|\geq\delta\}$
(we write $M^*(X)$ and
$M_\delta(X)$ to indicate dependence on $X$).
Call \(F \subset M\) \emph{rescaled \((t,\epsilon)\)-spanning set} if \(F \subset M^*\) and
\[
M_\epsilon \;\subset\; B^*(F,t,\epsilon)
\quad\text{where}\quad
B^*(F,t,\epsilon)\;=\;\bigcup_{x \in F} B^*(x,t,\epsilon).
\]
Clearly \(M_\epsilon\) is compact and \(B^*(x,t,\epsilon)\) is an open neighborhood of \(x\) for all \(x \in M^*\). So, the number
\[
R^*(t,\epsilon)
\;=\;
\min \,\bigl\{\,
\mathrm{card}(F)
:\,F \text{ is a rescaled \((t,\epsilon)\)-spanning set}
\bigr\}
\]
is finite for all \(t \ge 0\) and \(\epsilon > 0\), where \(\mathrm{card}(F)\) denotes the cardinality of \(F\) (when necessary we write
$R^*(t,\epsilon,X)$ to indicate dependence on $X$).
Then, we can introduce the following definition (equivalent to the one in the second author thesis \cite{r}):

\begin{definition}
\label{dino}
The \emph{rescaled topological entropy} of a \(C^1\) vector field \(X\) on a closed manifold \(M\) is defined by
\[
e^*(X)
\;=\;
\lim_{\epsilon \to 0}\,\limsup_{t \to \infty}
\frac{1}{t} \log R^*(t,\epsilon).
\]
\end{definition}

\subsection{Equivalent definitions}
In the sequel, we present some useful equivalent definitions.
For this we need Lemma 2.1 in \cite{wy} restated below.

\begin{lem}
\label{wendy}
For every $C^1$ vector field $X$ of a closed manifold $M$ there is $r_0>0$ such that
if $a\in M^*$ and $b\in M$ satisfy
\begin{equation}
\label{metanfeta}
d(a,b)\leq r_0\|X(a)\|\quad\mbox{ then }\quad \frac{1}2\|X(a)\|\leq \|X(b)\|\leq 2\|X(a)\|.
\end{equation}
\end{lem}

Next, we observe that since $d_t$ is a metric for all $t\geq0$, the dynamical $(t,\epsilon)$-balls satisfy the following property:
$$
y\in B(x,t,\frac{\epsilon}2)\implies B(x,t,\frac{\epsilon}2)\subset B(y,t,\epsilon).
$$
The lemma below is the analogous property for the rescaled dynamical $(t,\epsilon)$-balls.

\begin{lem}
\label{symmetry}
For every $C^1$ vector field $X$ of a closed manifold $M$
there is $r_0>0$ such that if $t\geq0$, $0<\epsilon<r_0$, $x\in M^*$ and
$$
y\in B^*(x,t,\frac{\epsilon}4)\quad\Longrightarrow\quad
B^*(x,t,\frac{\epsilon}4)\subset B^*(y,t,\epsilon).
$$
\end{lem}

\begin{proof}
Take $r_0$ as in Lemma \ref{wendy} and $y\in B^*(x,t,\frac{\epsilon}4)$.
Then, $\|X(y)\|\geq \frac{1}2\|X(x)\|>0$ and so $y\in M^*$.
Now, take $t\geq0$, $0<\epsilon<r_0$ and $w\in B^*(x,t,\frac{\epsilon}4)$.
Since $y\in B^*(x,t,\frac{\epsilon}4)$, $\|X(\varphi_s(x))\|\leq 2\|X(\varphi_s(y))\|$ for all $0\leq s\leq t$ by Lemma \ref{wendy} thus
$$
d(\varphi_s(y),\varphi_s(x))<\frac{\epsilon}2\|X(\varphi_s(y))\|,\quad\quad\forall 0\leq s\leq t.
$$
On the other hand, $w\in B^*(x,t,\frac{\epsilon}4)$ so $d(\varphi_s(x),\varphi_s(w))<\frac{\epsilon}4\|X(\varphi_s(x))\|$ for all $0\leq s\leq t$ thus
$$
d(\varphi_s(x),\varphi_s(w))<\frac{\epsilon}2\|X(\varphi_s(y))\|,\quad\quad\forall 0\leq s\leq t.
$$
Then,
\begin{eqnarray*}
d(\varphi_s(y),\varphi_s(w))&\leq& d(\varphi_s(y),\varphi_s(x))+d(\varphi_s(x),\varphi_s(w))\\
&<&\frac{\epsilon}2\|X(\varphi_s(y))\|+\frac{\epsilon}2\|X(\varphi_s(y))\|\\
&=&\epsilon\|X(\varphi_s(y))\|,\quad\quad\forall 0\leq s\leq t,
\end{eqnarray*}
proving $w\in B^*(y,t,\epsilon)$. Therefore, $B^*(x,t,\frac{\epsilon}4)\subset
B^*(y,t,\epsilon)$ ending the proof.
\end{proof}

Now, we give the first equivalent definition.
Let $X$ be a $C^1$ vector field of a closed manifold.
Given $t\geq0$ and $\epsilon>0$, we say that $F\subset M$ is
a {\em rescaled $(t,\epsilon,K)$-spanning set} if $F\subset M^*$ and
$K\subset B^*(F,t,\epsilon)$.
Let $R^*(t,\epsilon,K)$ be the minimal cardinality of a rescaled $(t,\epsilon,K)$-spanning set.
Define
\begin{equation}
\label{bomba}
e^*(X,K)=\lim_{\epsilon\to0}\limsup_{t\to\infty}\frac{1}t\log R^*(t,\epsilon,K).
\end{equation}

Then, we have the following lemma.

\begin{lem}
\label{fria}
For every $C^1$ vector field $X$ of a closed manifold one has
$$
e^*(X)=\sup_{K}e^*(X,K),
$$
where the supremum is over the compact subsets $K\subset M^*$.
\end{lem}

\begin{proof}
Fix a compact $K\subset M^*$. Then, there is $\epsilon_K>0$ such that $K\subset M_\epsilon$ for all $0<\epsilon<\epsilon_K$.
Suppose that $F$ is rescaled $(t,\epsilon)$-spanning for some $t>0$ and $0<\epsilon<\epsilon_K$.
Then, $F\subset M^*$ and $M_\epsilon\subset B^*(F,t,\epsilon)$ so
$F\subset M^*$ and $K\subset B^*(F,t,\epsilon)$ thus $F$ is rescaled $(t,\epsilon,K)$-spanning for all such $t$ and $\epsilon$.
It follows that $R^*(t,\epsilon,K)\leq R^*(t,\epsilon)$ for all such $t$ and $\epsilon$
hence
$$
e^*(X,K)=\lim_{\epsilon\to0}\limsup_{t\to\infty}\frac{1}t\log R^*(t,\epsilon,K)\leq \lim_{\epsilon\to0}\limsup_{t\to\infty}\frac{1}t\log R^*(t,\epsilon)=e^*(X)
$$
proving
$$
\sup_K e^*(X,K)\leq e^*(X).
$$
For the converse inequality we can assume $e^*(X)<\infty$ (otherwise the same argument will show that the supremum is infinity too).
Fix $\Delta>0$. Choose $\epsilon_0>0$ such that
$$
e^*(X)-\Delta<\limsup_{t\to\infty}\frac{1}t\log R^*(t,\epsilon),\quad\quad\forall 0<\epsilon<\epsilon_0.
$$
Given $0<\epsilon<\epsilon_0$ we have that $K_\epsilon=M_\epsilon$ is compact and contained in $M^*$. If $F$ is rescaled $(t,\epsilon,K_\epsilon)$-spanning for some $t>0$, then $F\subset M^*$ and $M_\epsilon\subset B^*(F,t,\epsilon)$ so $F$ is rescaled $(t,\epsilon)$-spanning too. This proves $R^*(t,\epsilon)\leq R^*(t,\epsilon,K_\epsilon)$. On the other hand, for fixed compact $K\subset M^*$ and $t>0$ the quantity $R^*(t,\epsilon,K)$ grows as $\epsilon\to0$ so can replace the limit as $\epsilon\to0$ by the supremum over $\epsilon>0$ in \eqref{bomba}.

So,
\begin{eqnarray*}
\limsup_{t\to\infty}\frac{1}t\log R^*(t,\epsilon)&\leq& \limsup_{t\to\infty}\frac{1}t\log R^*(t,\epsilon,K_\epsilon)\\
&\leq& \lim_{\gamma\to0}\limsup_{t\to\infty}\frac{1}t\log R^*(t,\gamma,K_\epsilon)\\
&=& e^*(X,K_\epsilon)\\
&\leq& \sup_K e^*(X,K).
\end{eqnarray*}
Letting $\epsilon\to0$ above we get
$$
e^*(X)\leq \sup_Ke^*(X,K)
$$
completing the proof.
\end{proof}

We use this lemma together with the arguments in Proposition 12 of Bowen \cite{b} and Lemma 2.5 in \cite{hw}  to prove the finiteness of the rescaled topological entropy.

\begin{lem}
\label{perseguido}
If $X$ is a $C^1$ vector field of a closed manifold $M$, then $e^*(X)<\infty$.
\end{lem}

\begin{proof}
It suffices to prove
\begin{equation}
\label{pi}
e^*(X)\leq 2d L
\end{equation}
for any Lipschitz constant $L$ of $X$, where $d$ is the dimension of $M$.

Indeed, fix a Lipschitz constant $L$.
By Lemma 2.5 in \cite{hw} one has
\begin{equation}
\label{cope}
e^{-Ls}\leq \frac{\|X(\varphi_s(z))\|}{\|X(z)\|}\leq e^{sL},\quad\quad\forall z\in M^*,\, s\geq0.
\end{equation}
If $L=0$, the right-hand side of \eqref{pi} is zero
while $\|X(\varphi_s(z))\|=\|X(z)\|$ for all $z\in M^*$ and $s\geq0$ by \eqref{cope}.
Then, for every compact $K\subset M^*$ and $\epsilon>0$, every rescaled $(0,\epsilon,K)$-spanning set is
rescaled $(t,\epsilon,K)$-spanning ($\forall t\geq0$) so $e^*(X,K)=0$
for all compact $K\subset M^*$. Thus, $e^*(X)=0$ by Lemma \ref{fria} and
\eqref{pi} holds.
Therefore, we can assume $L>0$.

Now, fix a compact subset $K\subset M^*$ and $\epsilon>0$.
Then, there are diffeomorphisms
$f_1,\cdots, f_n:B(0,2)\to M$ and a positive number $A$ such that
$$
K\subset \bigcup_{i=1}^nf_i(B(0,1))\subset \bigcup_{i=1}^nf_i(B(0,2))\subset M^*
\quad\mbox{ and }\quad
d(f_i(u),f_i(v))\leq A\|u-v\|,
$$
for all $u,v\in B(0,2)$ and $1\leq i\leq n$.
The third of the above inclusions implies that there is $\rho\in (0,\infty)$ such that
$$
\rho\leq \|X(x)\|,\quad\quad\forall x\in \bigcup_{i=1}^nf_i(B(0,2)).
$$

Now, for each $0<\delta\leq 2$ we let
$$
E(\delta)=\{(\delta l_1,\cdots,\delta l_d)\in \mathbb{R}^d\mid l_i\in \mathbb{Z},\, |l_i\delta|<2\}.
$$
Then, $card(E(\delta))\leq (\frac{5}\delta)^d$ and there is a constant $B>0$ (depending on the Euclidean metric of $\mathbb{R}^d$) such that for every $v\in B(0,1)$ there is $u\in E(\delta)$ satisfying $\|u-v\|\leq B\delta$. Replacing $B$ by $\frac{B}\rho$ we have
$\|u-v\|\leq B\rho\delta$ for all such $u,v$.

Since $L>0$, we can choose $T>0$ large such that
\begin{equation}
\label{pobre}
\frac{\epsilon}{e^{2Lt}AB}<2,\quad\quad\forall t\geq T.
\end{equation}
For all $t\geq T$ we define
$$
F=\bigcup_{i=1}^n f_i(E(\frac{\epsilon}{e^{2Lt}AB})).
$$
It follows that
\begin{eqnarray*}
card(F)&\leq & \sum_{i=1}^n card(f_i(E(\frac{\epsilon}{e^{2Lt}AB})))\\
&\overset{\eqref{pobre}}{\leq}& \left(\frac{5}{\left(\frac{\epsilon}{e^{2Lt}AB}\right)}\right)^dn\\
&=&
\left(\frac{5e^{2Lt}AB}{\epsilon}\right)^dn\\
&=& \left[\left(\frac{5AB}{\epsilon}\right)^dn\right]e^{2dLt},\quad\quad\forall t\geq T.
\end{eqnarray*}

On the other hand, given $y\in K$ one has $y=f_i(v)$ for some $v\in B(0,1)$ and $1\leq i\leq n$.
For this $v$ there is $u\in E(\frac{\epsilon}{e^{2Lt}AB})$ such that
$$
\|u-v\|\leq B\rho\frac{\epsilon}{e^{2Lt}AB}=\frac{\rho\epsilon}{e^{2Lt}A}.
$$
Thus, $x=f_i(u)$ belongs to $F$ and satisfies
$$
d(x,y)=d(f_i(u),f_i(v))\leq A\|u-v\|\leq A\rho\frac{\epsilon}{e^{2Lt}A}\leq \frac{\epsilon}{e^{2Lt}}\|X(x)\|.
$$
So, for all $t\geq T$ one has
$$
d(\varphi_s(x),\varphi_s(y))\leq e^{sL}d(x,y)\leq e^{tL}\frac{\epsilon}{e^{2Lt}}\|X(x)\|
\overset{\eqref{cope}}{\leq} \epsilon\|X(\varphi_s(x))\|,\quad\quad\forall 0\leq s\leq t.
$$
We conclude that $F$ is a rescaled $(t,\epsilon,K)$-spanning set for all $t\geq T$ so
$$
R^*(t,\epsilon,K)\leq card(F)\leq  \left[\left(\frac{5AB}{\epsilon}\right)^dn\right]e^{2dLt},\quad\quad\forall t\geq T.
$$
Then,
$$
\limsup_{t\to\infty}\frac{1}t\log R^*(t,\epsilon,K)\leq 2dL.
$$
Letting $\epsilon\to0$ we get $e^*(X,K)\leq 2dL$
and taking
the supremum over the compact subsets $K\subset M^*$ we get \eqref{pi}
from Lemma \ref{fria}. This completes the proof.
\end{proof}

Another equivalent definition is based on rescaled separating sets. More precisely, given a compact $K\subset M^*$, $t\geq0$ and $\epsilon>0$ we say that $E$ is a {\em rescaled $(t,\epsilon,K)$-separating set} if
$E\subset K$ and $B(x,t,\epsilon)\cap E=\{x\}$ for all $x\in E$.
Define
$$
S^*(t,\epsilon,K)=\max\{card(E)\mid E\mbox{ is a rescaled $(t,\epsilon,K)$-separating set}\},
$$
for all $t\geq0,\, \epsilon>0.$

\begin{lem}
\label{rasca}
For every $C^1$ vector field $X$ of a closed manifold $M$ one has
$$
e^*(X)=\sup_K\lim_{\epsilon\to0}\limsup_{t\to\infty}\frac{1}t\log S^*(t,\epsilon,K),
$$
where the supremum is over the compact subsets $K\subset M^*$.
\end{lem}

\begin{proof}
First we prove that for every compact $K\subset M^*$ there is
$\epsilon_K>0$ such that
$$
S^*(t,\epsilon,K)\leq R^*(t,\frac{\epsilon}4),\quad\quad\forall t>0,\, 0<\epsilon<\epsilon_K.
$$
Fix a compact $K\subset M^*$.
Then, there is $0<\epsilon_K<r_0$ such that
$$
K\subset M_{\frac{\epsilon}4},\quad\quad\forall 0<\epsilon <\epsilon_K.
$$
We can assume without loss of generality that $\epsilon_K<r_0$, where
$r_0$ is given by Lemma \ref{symmetry}.

Let $E$ and $F$ be a rescaled $(t,\epsilon,K)$-separating set and a rescaled $(t,\frac{\epsilon}4)$-spanning set respectively.
Then, $E\subset K\subset M_{\frac{\epsilon}4}$ and so there is a map $\phi:E\to F$ such that
$$
x\in B^*(\phi(x),t,\frac{\epsilon}4),\quad\quad\forall x\in E.
$$
Now, suppose $\phi(x)=\phi(x')$ for some $x,x'\in E$.
Then, the common value $z=\phi(x)=\phi(x')$ satisfies
$x,x'\in B^*(z,t,\frac{\epsilon}4)$. Since $\epsilon<r_0$,
$B^*(z,t,\frac{\epsilon}4)\subset B^*(x,t,\epsilon)$ by Lemma \ref{symmetry}
so $x'\in B^*(x,t,\epsilon)\cap E=\{x\}$ hence $x'=x$.
It follows that $\phi$ is injective thus $card(E)\leq card(F)$. Since $E$ and $F$ are arbitrary, $S^*(t,\epsilon,K)\leq R^*(t,\frac{\epsilon}4)$. It follows that
\begin{equation}
\label{veneno}
S^*(t,\epsilon,K)<\infty,\quad\forall t\geq0,\, 0<\epsilon< r_0\mbox{ and all compact }K\subset M^*.
\end{equation}
Moreover,
$$
\limsup_{\epsilon\to0}\limsup_{t\to\infty}\frac{1}t\log S^*(t,\epsilon,K)\leq \lim_{\epsilon\to0}\limsup_{t\to\infty}\frac{1}t\log R^*(t,\frac{\epsilon}4)=e^*(X).
$$
Since $K\subset M^*$ is arbitrary,
$$
\sup_K\limsup_{\epsilon\to0}\limsup_{t\to\infty}\frac{1}t\log S^*(t,\epsilon,K)\leq e^*(X).
$$

To prove the reversed inequality we first prove that there is $\epsilon_0>0$ such that
for every $0<\epsilon<\epsilon_0$ there is $K_\epsilon\subset M^*$ compact such that
\begin{equation}
\label{metro}
 R^*(t,\epsilon)\leq S^*(t,\frac{\epsilon}4,K_\epsilon),\quad\quad\forall t>0.
\end{equation}
Indeed, just take $\epsilon_0=r_0$ from Lemma \ref{symmetry} and $K_\epsilon=M_\epsilon$ for the given $0<\epsilon<\epsilon_0$.
Let $E$ be a $(t,\frac{\epsilon}4,K_\epsilon)$-separating set of maximal cardinality (which is finite by \eqref{veneno}). We shall prove that $E$ is rescaled $(t,\epsilon)$-spanning i.e.
\begin{equation}
\label{rumble}
y\in B^*(E,t,\epsilon),
\end{equation}
for all $y\in K_\epsilon$.

Take $y\in K_\epsilon$.
If $y\in E$ so \eqref{rumble} holds hence we can assume
$y\notin E$. It follows that $card(E\cup \{y\})=card(E)+1>card(E)$.
Since $E\cup \{y\}\subset K_\epsilon$,
we conclude that $E\cup \{y\}$ is not rescaled $(t,\frac{\epsilon}4,K_\epsilon)$-separating. Thus, there is $z\in E\cup \{y\}$ such that
$$
(B^*(z,t,\frac{\epsilon}4)\cap E)\cup (B^*(z,t,\frac{\epsilon}4)\cap \{y\})\neq \{z\}.
$$

If $z\in E$, then $B^*(z,t,\frac{\epsilon}4)\cap E=\{z\}$ so $B^*(z,t,\frac{\epsilon}4)\cap \{y\}\neq\emptyset$ thus
$y\in B^*(z,t,\frac{\epsilon}4)$ hence \eqref{rumble} holds.

If $z\notin E$, then $z=y$ so $B^*(z,t,\frac{\epsilon}4)\cap \{y\}=\{z\}$ thus
$B^*(y,t,\frac{\epsilon}4)\cap E\neq\emptyset$ yielding
$x\in B^*(y,t,\frac{\epsilon}4)$ for some $x\in E$.
Since $\epsilon<r_0$, $B^*(y,t,\frac{\epsilon}4)\subset B(x,t,\epsilon)$
Lemma \ref{symmetry} and since $y\in B^*(y,t,\frac{\epsilon}4)$ we get
$y\in B^*(x,t,\epsilon)$  thus \eqref{rumble} holds.
Then, $E$ is rescaling $(t,\epsilon)$-spanning proving \eqref{metro}.

So,
\begin{eqnarray*}
\limsup_{t\to\infty}\frac{1}t\log R^*(t,\epsilon)&\leq& \limsup_{t\to\infty}\frac{1}t\log S^*(t,\frac{\epsilon}4,K_\epsilon)\\
&\leq& \lim_{\gamma\to0}\limsup_{t\to\infty}\frac{1}t\log S^*(t,\gamma,K_\epsilon)\\
&\leq & \sup_K\lim_{\gamma\to0}\limsup_{t\to\infty}\frac{1}t\log S^*(t,\gamma,K)
\end{eqnarray*}
Letting $\epsilon\to0$ above we get
$$
e^*(X)\leq \sup_K\lim_{\gamma\to0}\limsup_{t\to\infty}\frac{1}t\log S^*(t,\gamma,K)
$$
completing the proof.
\end{proof}

\subsection{Characterizing positive rescaled topological entropy}
Let $X$ be a $C^1$ vector field of a closed manifold $M$. It can be proved that $e(X)>0$ if and only if there are sequences
$t_n\to\infty$ (as $n\to\infty$) and $E_n\subset X$ such that
\begin{equation}
\label{esforzo}
\limsup_{n\to\infty}\frac{1}{t_n}\log card(E_n)>0\quad\mbox{ and }
\quad
\inf_{n\in\mathbb{N}}\inf_{p\neq q\in E_n}d_{t_n}(p,q)>0.
\end{equation}

The result of this section is an analogous fact for the rescaled topological entropy:

\begin{prop}
\label{naro}
Let $X$ be a $C^1$ vector field of a closed manifold $M$.
Then, $e^*(X)>0$ if and only if
there are a compact $K\subset M^*$ and sequences $t_n\to\infty$, $E_n\subset K$ such that
\begin{equation}
\label{sambada}
\limsup_{n\to\infty}\frac{1}{t_n}\log card(E_n)>0\quad\mbox{ and }
\quad
\inf_{n\in\mathbb{N}}\inf_{p\neq q\in E_n}d^*_{t_n}(p,q)>0.
\end{equation}
\end{prop}

\begin{proof}
By Lemma \ref{rasca} if $e^*(X)>0$  there are a compact $K\subset M^*$, $\epsilon>0$ and a sequence $t_n\to\infty$ such that
$$
\limsup_{n\to\infty}\frac{1}{t_n}\log S^*(t_n,\epsilon,K)>0.
$$
For all $n\in\mathbb{N}$ we let $E_n$ be a rescaled $(t_n,\epsilon,K)$-separating set with cardinality $S^*(t_n,\epsilon,K)$. Then, $E_n\subset K$ for all $n\in\mathbb{N}$ and
$$
\limsup_{n\to\infty}\frac{1}{t_n}\log card(E_n)>0.
$$
Moreover,
for all $n\in\mathbb{N}$ and all distinct $p,q\in E_n$ one has $d^*_{t_n}(p,q)\geq \epsilon$ yielding
$$
\inf_{n\in\mathbb{N}}\inf_{p\neq q\in E_n}d^*_{t_n}(p,q)\geq\epsilon>0.
$$
Therefore, \eqref{sambada} holds.

Conversely, suppose that there are $K$, $t_n$ and $E_n\subset K$ satisfying \eqref{sambada}.
The second inequality in \eqref{sambada} implies that there is $\epsilon>0$ such that $E_n$ is rescaled $(t_n,\epsilon,K)$-separating for every $n\in\mathbb{N}$.
So, $S^*(t_n,\epsilon,K)\geq card(E_n)$ for all $n\in\mathbb{N}$ thus Lemma \ref{rasca} implies
$$
e^*(X)\geq \limsup_{n\to\infty}\frac{1}{t_n}\log S^*(t_n,\epsilon,K)\geq \limsup_{n\to\infty}\frac{1}{t_n}\log card(E_n)>0
$$
completing the proof.
\end{proof}

A related remark is as follows.

\begin{remark}
\label{ba}
Since surface vector fields have zero toplogical entropy \cite{lg}, such vector fields do not exhibit sequences \( t_n \) and \( E_n \) satisfying \eqref{esforzo}.
However, an example of a vector field on the two-torus exhibiting a compact set \( K \), a sequence \( t_n \), and subsets \( E_n \subset K \) satisfying \eqref{sambada} will be constructed in the proof of the theorem. The singularities will play a fundamental role
in this construction.
\end{remark}

\section{Rescaled metric entropy}
\label{hop}

\noindent
In this section, we present the necessary tools to prove Item (1) of the theorem.
First, we will localize the rescaled metric entropy as follows.
Given $K\subset M$ and $t,\epsilon,\delta>0$ we
say that $F\subset M^*$ is a {\em rescaled $\mu$-$(t,\epsilon,\delta)$-spanning set of $K$} if $\mu(K\setminus B^*(F,t,\epsilon))<\delta.$
We need the following lemma for the forthcoming definition.

\begin{lem}
\label{paja}
Let $X$ a $C^1$ vector field $X$ of a closed manifold $M$ and $\mu$ be a Borel probability measure of $M$. If $\mu(Sing(X))=0$, then for all compact subset $K\subset M$ and $t,\epsilon,\delta>0$ there exists a finite rescaled $\mu$-$(t,\epsilon,\delta)$-spanning set of $K$.
\end{lem}

\begin{proof}
Fix $K\subset M$ compact and $t,\epsilon,\delta>0$.
Since $\mu(Sing(X))=0$, there is an open neighborhood $U$ of $Sing(X)$ such that $\mu(U)<\delta$.
Clearly $M\setminus U$ is compact and contained in $M^*$ so there is a finite set $F\subset M\setminus U\subset M^*$ such that
$$
M\setminus U\subset B^*(F,t,\epsilon)\quad\mbox{ and so }\quad
M\setminus B^*(F,t,\epsilon)\subset U.
$$
Then,
$$
\mu(K\setminus B^*(F,t,\epsilon))\leq \mu(M\setminus B^*(F,t,\epsilon))\leq \mu(U)<\delta
$$
and so $F$ is a rescaled $\mu$-$(t,\epsilon,\delta)$-spanning set of $K$.
\end{proof}

It follows from this lemma that if $\mu(Sing(X))=0$, then
\begin{multline*}
R^{\mu*}(t,\epsilon,\delta,K)=\min\{card(F):\\
F\mbox{ is a rescaled $\mu$-$(t,\epsilon,\delta)$-spanning set of }K\}<\infty,
\end{multline*}
for all compact subset $K\subset M$ and $t,\epsilon,\delta>0$.
Then, we introduce the following auxiliary definition.

\begin{definition}
For every $C^1$ vector field $X$ of a closed manifold $M$, $K\subset M$ compact and every Borel probability measure $\mu$ of $M$ with $\mu(Sing(X))=0$ we define
$$
e^*_\mu(X,K)=\lim_{\delta\to0}\lim_{\epsilon\to0}\limsup_{t\to\infty}\frac{1}t\log R^{\mu*}(t,\epsilon,\delta,K).
$$
\end{definition}

Since $R^{\mu*}(t,\epsilon,\delta,K)$ decreases as $\epsilon\to0$ (resp. $\delta\to0$), one has
\begin{equation}
\label{gatuso}
e^*_\mu(X,K)=\sup_{\delta>0}\sup_{\epsilon>0}\limsup_{t\to\infty}\frac{1}t\log R^{\mu*}(t,\epsilon,\delta,K).
\end{equation}

We estimate this entropy by noting that $R^{\mu*}(t,\epsilon,\delta,M)$ is just $N^*_\mu(t,\epsilon,\delta)$ for all $t,\epsilon>0$ and $0<\delta<1$.
From this we obtain the following lemma.

\begin{lem}
\label{babel}
Let $X$ a $C^1$ vector field $X$ of a closed manifold $M$ and $\mu$ be a Borel probability measure of $M$. If $\mu(Sing(X))=0$, then
$e^*_\mu(X)=e^*_\mu(X,M)$.
\end{lem}

An advantage of the "local" entropy $e^*_\mu(X,K)$ is given by the following lemma.

\begin{lem}
\label{short}
For every $C^1$ vector field $X$ of a closed manifold $M$ and every compact subset $K\subset M^*$ one has
$$
\sup_\mu e^*_\mu(X,K)\leq e^*(X,K),
$$
where the supremum is over the Borel probability measures $\mu$ with $\mu(Sing(X))=0$.
\end{lem}

\begin{proof}
Fix a compact $K\subset M^*$ and a Borel probability measure $\mu$ of $M$ such that $\mu(Sing(X))=0$.
If $F$ is a rescaled $(t,\epsilon,K)$-spanning set, then
$K\subset B^*(F,t,\epsilon)$ and so $\mu(K\setminus B^*(F,t,\epsilon))=0<\delta$
for all $\delta>0$.
Thus, $F$ is a $\mu$-$(t,\epsilon,\delta)$-spanning set of $K$
proving
$$
R^{\mu*}(t,\epsilon,\delta,K)\leq R^*(t,\epsilon,K),\quad\quad\forall t,\epsilon,\delta>0.
$$
Then, by taking $\log$, dividing by $t$ and letting $\epsilon,\delta\to0$ we get
$e^*_\mu(X,K)\leq e^*(X,K)$ and we are done.
\end{proof}

Another advantage is that the rescaled metric entropy $e^*_\mu(X)$ splits into the local entropies $e^*_\mu(X,K)$ likewise $e^*(X)$ splits into the local entropies $e^*(X,K)$ according to Lemma \ref{fria}. More precisely, we have the following lemma.

\begin{lem}
\label{singar}
Let $X$ a $C^1$ vector field $X$ of a closed manifold $M$ and $\mu$ be a Borel probability measure of $M$. If $\mu(Sing(X))=0$, then
$$
e^*_\mu(X)=\sup_Ke^*_\mu(X,K),
$$
where the supremum is over the compact subsets $K\subset M^*$.
\end{lem}

\begin{proof}
By Lemma \ref{babel} it suffices to show
$$
e^*_\mu(X,M)=\sup_Ke^*_\mu(X,K)
$$
where the supremum is over the compact $K\subset M^*$.

Clearly $K\subset K'$ implies $R^{\mu*}(t,\epsilon,\delta,K)\leq
R^{\mu*}(t,\epsilon,\delta,K')$. Since $M$ itself is compact, we obtain
$e^*_\mu(X,K)\leq e^*_\mu(X,M)$ so
$$
\sup_Ke^*_\mu(X,K)\leq e^*_\mu(X,M)
$$
where the supremum is over the compact subsets $K\subset M^*$.

To prove the reversed inequality we can assume $e^*_\mu(X,M)<\infty$
(otherwise a similar argument shows that the supremum is $\infty$ too).
Fix $\Delta>0$. Then, there are $\epsilon,\delta>0$ such that
\begin{equation}
\label{inmortal}
e^*_\mu(X,M)-\Delta<\limsup_{t\to\infty}\frac{1}t\log R^{\mu*}(t,\epsilon,\delta,M).
\end{equation}
Since $\mu(Sing(X))=0$, there exists an open neighborhood $U$ of $Sing(X)$ such that
$$
\mu(U)<\frac{\delta}2.
$$
Take $K'=M\setminus U$ hence $K'$ is compact and $K'\subset M^*$.

Let $F$ be any rescaled $\mu$-$(t,\epsilon,\frac{\delta}2)$-spanning set of $K'$.
Then,
$
\mu(K'\setminus B^*(F,t,\epsilon))<\frac{\delta}2
$
and so
\begin{eqnarray*}
\mu(M\setminus B^*(F,t,\epsilon))&\leq& \mu(M\setminus K')+\mu(K'\setminus B^*(F,t,\epsilon))\\
&=&\mu(U)+\mu(K'\setminus B^*(F,t,\epsilon))\\
&<&\frac{\delta}2+\frac{\delta}2\\
&=& \delta.
\end{eqnarray*}
It follows that $F$ is a rescaled $\mu$-$(t,\epsilon,\delta)$-spanning set of $M$.
Therefore,
$$
R^{\mu*}(t,\epsilon,\delta,M)\leq R^{\mu*}(t,\epsilon,\frac{\delta}2,K'),\quad\quad\forall t>0.
$$
Then, by taking $\log$, dividing by $t$ and letting $t\to\infty$ using \eqref{inmortal} we obtain
$$
e^*_\mu(X,M)-\Delta\leq \limsup_{t\to\infty}\frac{1}t\log R^{\mu*}(t,\epsilon,\frac{\delta}2,K')\overset{\eqref{gatuso}}{\leq} e^*_\mu(X,K').
$$
Since $K'\subset M^*$ is compact, we get
$$
e^*_\mu(X,M)-\Delta<\sup_K e^*_\mu(X,K)
$$
where the supremum is over the compact $K\subset M\setminus  Sing(X)$. Letting $\Delta\to0$ we get
$$
e^*_\mu(X,M)\leq\sup_K e^*_\mu(X,K)
$$
completing the proof.
\end{proof}

\section{Proof of the theorem}
\label{hup}

\noindent
To any $C^1$ vector field $X$ of a closed manifold $M$
we can assign the nonnegative value $e^*(X)$ from Definition \ref{dino}.
By Lemma \ref{perseguido} we have $e^*(X)\in [0,\infty)$.
We shall prove that this number satisfies the required properties.

First we show Item (1).
The proof follows from the sequence of equalities and inequalities below where the suprema are over the compact subsets $K\subset M^*$ and the Borel probability measures $\mu$ of $M$ with $\mu(Sing(X))=0$ respectively:
\begin{eqnarray*}
\sup_\mu e^*_\mu(X)&=&\sup_\mu\sup_K e^*_\mu(X,K)\quad\quad\mbox{(by Lemma \ref{short})}\\
&=&\sup_K\sup_\mu e^*_\mu(X,K)\\
&\leq& \sup_K e^*(X,K)\quad\quad\quad\quad\mbox{(by Lemma \ref{singar})}\\
&=&e^*(X)\quad\quad\quad\quad\quad\quad\quad\mbox{(by Lemma \ref{fria}).}
\end{eqnarray*}

Afterwards, we prove Item (2).
Recall that a Borel probability measure $\mu$ of $M$ is
{\em invariant} for $X$ if $\mu\circ\varphi_t=\mu$ for all $t\in\mathbb{R}$.
We say that $\mu$ is {\em ergodic} for $X$ if $\mu(I)\in\{0,1\}$ for all measurable subset $I$ which is invariant (i.e. $\varphi_t(I)=I$ for all $t\in \mathbb{R}$).

By Theorem A in \cite{s} we have
\begin{equation}
\label{gigante}
e(X)=\sup\{e_\mu(X):\mu\mbox{ is ergodic invariant for } X\},
\end{equation}
where
$$
e_\mu(X)=\lim_{\epsilon\to0}\limsup_{t\to\infty}\frac{1}t\log R^\mu(t,\epsilon,\delta),\quad\quad(\forall 0<\delta<1)
$$
and $R^\mu(t,\epsilon,\delta)$ is the minimal number of dynamical $(t,\epsilon)$-balls \eqref{kubrick} needed to cover a set of $\mu$-measure greater that $1-\delta$.

We can improve \eqref{gigante} as follows:
If $\mu$ is ergodic invariant with $\mu(Sing(X))>0$, then
$\mu$ is the Dirac measure supported on a singularity. Then, $e_\mu(X)=0$ for all such measures thus \eqref{gigante} reduces to
\begin{multline}
\label{yoyo}
e(X)=\sup\big\{e_\mu(X):\mu\mbox{ is ergodic invariant for }X
\mbox{ with }\mu(Sing(X))=0\big\}.
\end{multline}
Next, we have
$$
e_\mu(X)\leq e^*_\mu(X)
$$
for every ergodic invariant measure $\mu$.
Indeed, define
\begin{equation}
\label{pqp}
\|X\|_\infty=\sup_{x\in M}\|X(x)\|.
\end{equation}
Since $M$ is compact and $X$ is nonzero, $\|X\|_\infty\in (0,\infty)$.
Given $0<\delta<1$ and $t,\epsilon>0$ if a collection of rescaled dynamical $(t,\epsilon)$-balls
$$
\{y\in M:d(\varphi_s(x),\varphi_s(y))<\epsilon\|X(\varphi_s(x))\|,\,\forall 0\leq s\leq t\}
$$
covers a subset of measure greater than $1-\delta$, then the associated collection of
dynamical $(t,\epsilon\|X\|_\infty)$-balls
$$
\{y\in M:d(\varphi_s(x),\varphi_s(y))<\epsilon\|X\|_\infty,\,\forall 0\leq s\leq t\}
$$
also covers a subset of measure greater than $1-\delta$.
Since both collections have the same cardinality, we obtain
$R^\mu(t,\epsilon\|X\|_\infty,\delta)\leq R^{\mu*}(t,\epsilon,\delta)$ for all $0<\delta<1$ and $t,\epsilon>0$ proving the assertion.

Therefore,
\begin{eqnarray*}
e(X)&\overset{\eqref{yoyo}}{=}&\sup\{e_\mu(X):\mu\mbox{ is ergodic invariant for }X\mbox{ with }\mu(Sing(X))=0\}\\
&\leq& \sup\{e^*_\mu(X):\mu(Sing(X))=0\}\\
&\leq&e^*(X)\quad\quad\quad\quad\mbox{ (by Item (1))}.
\end{eqnarray*}
This proves Item (2).

Next, we prove Item (3).
Suppose that $X$ is nonsingular.
Then,
$$
m(X)=\inf_{x\in M}\|X(x)\|\in (0,\infty).
$$
It follows that
$B(F,t,\epsilon)\subset B^*(F,t,\frac{\epsilon}{m(X)})$ for all $t\geq0$, $\epsilon>0$ and $F\subset M^*$.
Hence, every $(t,\epsilon)$-spanning set is rescaled $(t,\frac{\epsilon}{m(X)},\delta)$-spanning for all $t\geq0$ and $\delta,\epsilon>0$ thus
$R^*(t,\epsilon,\delta)\leq R(t,m(X)\epsilon)$ for all such $t,\delta,\epsilon$ yielding
$e^*(X)\leq e(X)$ and so $e^*(X)=e(X)$ by Item (2).
This proves Item (3).

In the sequel, we prove Item (4).
Take a periodic flow $\psi$ on $T^2$ with constant velocity $1$. We see $T^2$ as a square $S$ with sides of length $4$ in $\mathbb{R}^2$ with vertices at the points $(-2,0),(2,0),(-2,4)$ and $(2,4)$ and identifying first the lower and upper sides and then identifying the right and left sides (see Figure \ref{fig1}). Let us consider on $S$ the vector field $X^0$ with constant and equal to $(1,0)$ velocity. Thus $X^0$ generates the desired $\psi$.

Next, consider a $C^\infty$ function $\rho:[-2,2]\to \mathbb{R}$ satisfying the following conditions:
\begin{enumerate}
\item
$\rho=1$ in $[-2,-1]\cup[1,2]$;
\item
$\rho=-x$ in $[-\frac{1}4,\frac{1}4]$;
\item
$\rho=x-\frac{2}3$ in $[\frac{7}{12},\frac{9}{12}]$.
\end{enumerate}
Define the $C^\infty$ vector field $X$ of $T^2$ by $X(p)=\rho(x) X^0(p)$ for all $p=(x,y)\in T^2$. The portrait face of $X$ is depicted in Figure \ref{fig1}.
As usual $\varphi$ denotes the flow generated by $X$.
Note that $Sing(X)=\{0,\frac{2}3\}\times [0,4]$.

\begin{figure}
  \includegraphics[width=150pt]{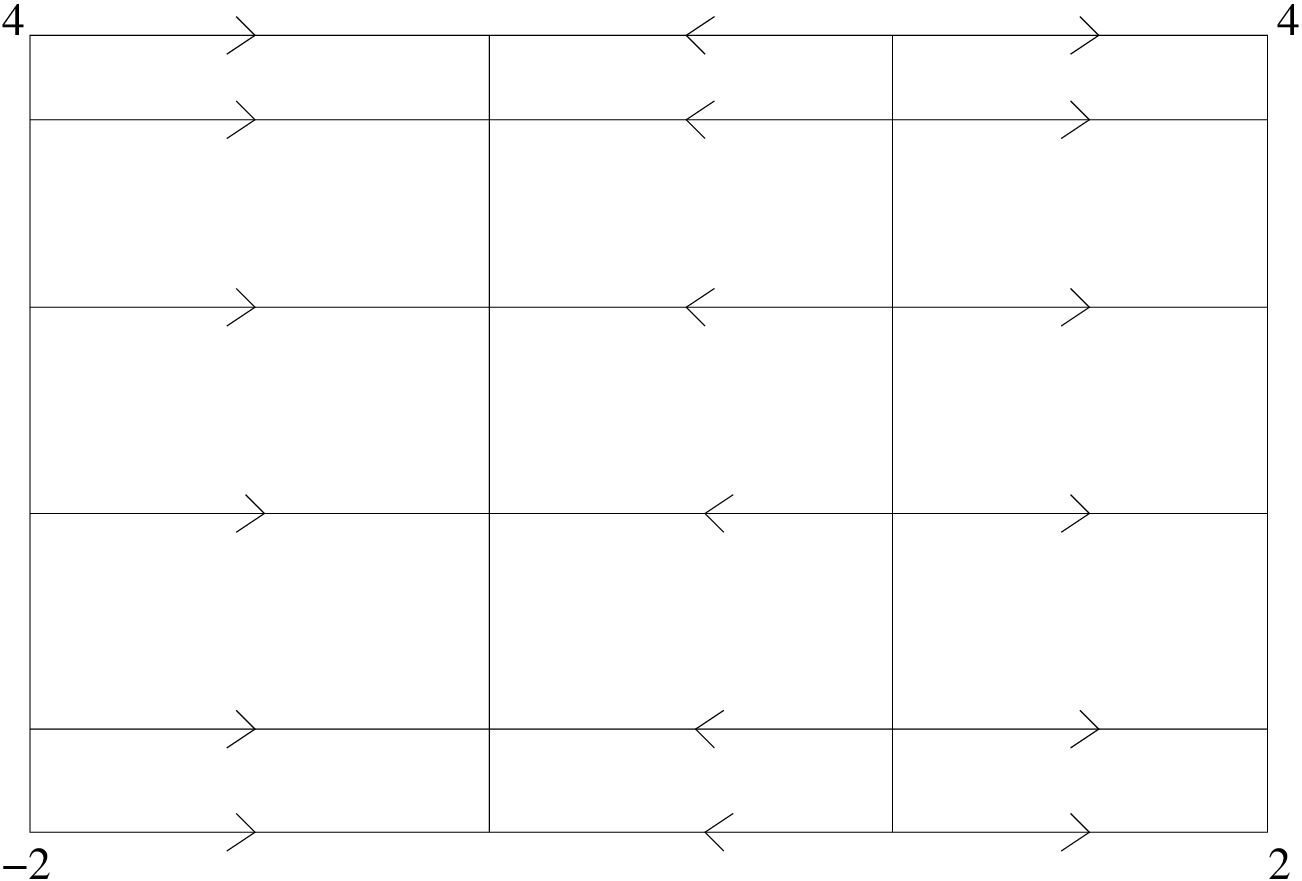}
 \caption{Portrait face of $X$}
  \label{fig1}
\end{figure}
Now, consider the circle $C$ in $T^2$ represented by $\{-2\}\times [0,4]$.
Then, $C$ is compact and $C\subset M\setminus Sing(X)$.
Moreover,
by the choice of $\rho$ we have $\|X(\varphi_n(p))\|=e^{-n}$ whenever $p\in C$.
For each $n\in\mathbb{N}$ we divide $C$ into $2^n$ segments of equal length. The dividing points will be collected together in a subset $E_n\subset C$ with $card(E_n)=2^n$. It follows from the construction of $X$ that $d(\varphi_n(p),\varphi_n(q))\geq L2^{-n}$ for every $n\in\mathbb{N}$ and all distinct points $p,q\in E_n$, where $L$ denotes the length of $C$.
Then, the sequence $t_n=n$ satisfies
$$
\inf_{n\in\mathbb{N}}\inf_{p\neq q\in E_n}d^*_n(p,q)
\geq \inf_{n\in\mathbb{N}}L\left(\frac{e}2\right)^n=\frac{Le}2>0.
$$
Since
$$
\limsup_{n\to\infty}\frac{1}n\log card(E_n)=\log 2>0,
$$
one has $e^*(X)>0$ by Proposition \ref{naro}.

Next, we prove Item (5).
Suppose that $X$ is rescaled topological conjugated to another vector field $Y$ (with flow $\varphi$) of a closed manifold $N$.
More precisely,
there is a rescaled homeomorphism $h:N\to M$ such that $\phi_t\circ h=h\circ \varphi_t$ for all $t\in\mathbb{R}$.

In particular, $h(N^*)=M^*$.
Fix $\epsilon>0$.
Then, there is $\delta>0$ such that
$$
y,y'\in N\quad\mbox{ and }\quad d(y,y')\leq \delta\|Y(y)\|\implies d(h(y),h(y'))\leq \epsilon\|X(h(y))\|.
$$
Then, if $F$ is rescaling $(t,\delta)$-spanning for $Y$,
$h(F)$ is rescaling $(t,\epsilon)$-spanning for $X$ thus
$R^*(t,\epsilon,X)\leq R^*(t,\delta,Y)$ for all $t\geq0$.
From this we obtain $e^*(X)\leq e^*(Y)$.
Reversing the roles of $X$ and $Y$ we obtain $e^*(Y)\leq e^*(X)$
proving the result.

Finally, we prove Item (6).
Suppose that $Sing(X)$ is dynamically isolated.
Then, there is $\delta>0$ such that every periodic orbits of $X$ intersects
$M_\delta$.
We now follow closely the proof of Theorem 5 in \cite{bw} with the aid of Theorem 1.1 in \cite{wy}.

Given $t,\beta>0$ let $v_\beta(t)$ denote the number of different periodic orbits with periods belonging to the closed interval $[t-\beta,t+\beta]$.
Let $\alpha>0$ be given by Theorem 1.1-(v) in \cite{wy} for $\epsilon=1$.

We claim that
\begin{equation}
\label{poro}
v_{\frac{\alpha}2}(t)\leq S^*(t,\alpha,M_\delta),\quad\quad\forall t>0.
\end{equation}

Indeed, fix $t>0$.
By selecting one point for each periodic orbit whose period belongs to the closed interval $[t-\frac{\alpha}2,t+\frac{\alpha}2]$ we form a subset $E\subset M^*$ such that $card(E)=v_{\frac{\alpha}2}(t)$.
Since every periodic orbit intersects $M_\delta$, we can further assume that $E\subset M_\delta$.

Let us prove that $E$ is rescaled $(t,\alpha,M_\delta)$-separating.
Otherwise, there would exist distinct $x,y\in E$ such that
$$
\frac{d(\varphi_s(x),\varphi_s(y))}{\|X(\varphi_s(x))\|}<\alpha,\quad\quad\forall 0\leq s\leq t.
$$
Let $a$ and $b$ be the periods of $x$ and $y$ respectively so
$a,b\in [t-\frac{\alpha}2,t+\frac{\alpha}2]$, $\varphi_a(x)=x$ and $\varphi_b(y)=y$.
Define $m=[\frac{t-\frac{\alpha}2}\alpha]$ and the sequences
$(t_i)_{i\in\mathbb{Z}},(u_i)_{i\in\mathbb{Z}}$ by
$$
t_i=pa+q\alpha\quad\mbox{ and }\quad u_i=pb+q\alpha
$$
whenever $i=pm+q$ for some $p\in \mathbb{Z}$ and $0\leq q<m.$
Since $0\leq q\alpha\leq t$ for $0\leq q<m$,
it follows that
$$
\frac{d(\varphi_{t_i}(x),\varphi_{u_i}(y))}{\|X(\varphi_{t_i}(x))\|}=\frac{d(\varphi_{q\alpha}(x),\varphi_{q\alpha}(y))}{\|X(\varphi_{q\alpha}(x))\|}<\alpha,\quad\quad\forall i\in\mathbb{Z}.
$$
Then, Theorem 1.1-(v) in \cite{wy} provides $t\in [-1,1]$ such that $\varphi_{u_0}(y)=\varphi_t(\varphi_{t_0}(x))$. It follows that
$x$ and $y$ belong to the same orbit. However, $x$ and $y$ are distinct and so they are in different orbits by construction, a contradiction. This contradiction proves that $E$ is $(t,\alpha,M_\delta)$-separating. It follows that
$card(E)\leq S^*(t,\alpha,M_\delta))$ and, since $card(E)=v_{\frac{\alpha}2}(t)$, we get \eqref{poro}.

It follows that
$$
v(t)\leq \sum_{n=1}^{[\frac{t}\alpha]}v_{\frac{\alpha}2}(n\alpha)\overset{\eqref{poro}}{\leq} \sum_{n=1}^{[\frac{t}\alpha]} S^*(n\alpha,\alpha,M_\delta),\quad\quad\forall t\geq0.
$$
Now, $n\alpha\leq t$ whenever $n=1,2,\cdots, [\frac{t}\alpha]$
and $S^*(s,\alpha,M_\alpha)$ does not decreases as $s$ increases so
$S^*(n\alpha,\alpha,M_\delta)\leq S^*(t,\alpha,M_\delta)$ for all such $n$'s thus
$$
v(t)\leq \frac{t}\alpha S^*(t,\alpha,N_\delta),\quad\quad\forall t\geq0.
$$
Therefore,
\begin{eqnarray*}
\limsup_{t\to\infty}\frac{1}t\log v(t)&\leq& \limsup_{t\to\infty}\left(\frac{\log t}t-\frac{\log\alpha}t+\frac{1}t\log S^*(t,\alpha,M_\delta)\right)\\
&=&\limsup_{t\to\infty}\frac{1}t\log S^*(t,\alpha,M_\delta)\\
&\leq& e^*(X) \quad\quad\quad\quad\quad\quad\mbox{(by Lemma \ref{rasca})}
\end{eqnarray*}
completing the proof.
\qed

\medskip

We finish with a commentary. An
anonymous referee raised the concern that the example presented in the above proof appears somewhat artificial. In response, we present one more example of positive rescaling topological entropy in $S^2$:

\begin{ex}
\label{eno}
The $C^\infty$ vector field of $S^2$ defined by
$$
X^0(x,y,z)=
\left\{\begin{array}{rcl}
(x^2+y^2)e^{-\frac{1}{x^2+y^2}}(xz,yz,-x^2-y^2), & \mbox{ if } (x,y)\neq (0,0)\\
& & \\
0, &\mbox{ otherwise}
\end{array}
\right.
$$
can be $C^\infty$ approximated by ones with positive rescaled topological entropy.
\end{ex}

\begin{proof}
Note that
$$
X^0(x,y,z)=\rho_0(z)(xz,yz,-x^2-y^2),\quad\quad\forall (x,y,z)\in S^2,
$$
where $\rho_0:\mathbb{R}\to\mathbb{R}$ is the $C^\infty$ function
$$
\rho_0(z)=
\left\{\begin{array}{rcl}
(1-z^2)e^{-\frac{1}{1-z^2}}, & \mbox{ if } & z\in (-1,1)\\
0, & \mbox{ if } & z\in \{-1,1\}.
\end{array}
\right.
$$
The graph of $\rho_0$ is depicted in the left-hand side of Figure \ref{fig2}.
Note that it is tangent to the $z$ axis at $z=\pm 1$. From this
we can $C^\infty$ approximated it by $\rho(z)$ having two additional zeroes $-1<b<a<0<1$ as indicated in the right-hand side of that figure.

Define the $C^\infty$ vector field $X$ of $S^2$ by
$$
X(x,y,z)=\rho(z)(xz,yz,-x^2-y^2),\quad\quad\forall (x,y,z)\in S^2.
$$

Clearly $X$ is $C^\infty$ close to $X^0$.
Its dynamics can be described as follows: First of all
$$
Sing(X)=\{\sigma_+,\sigma_-\}\cup\{(x,y,a):x^2+y^2=1-a^2\}\cup\{(x,y,b):x^2+y^2=1-b^2\}
$$
where $\sigma_\pm=(0,0,\pm 1)$ are the poles of $S^2$.
All trajectories are contained in vertical planes.
Those over the plane $z=a$ go from $\sigma_+$ to one of the equilibrium points in $\{(x,y,a):x^2+y^2=1-a^2\}$; those in between the planes $z=a$ and $z=b$ go from one equilibrium in $\{(x,y,a):x^2+y^2=1-a^2\}$
to one in $\{(x,y,a):x^2+y^2=1-b^2\}$ and, finally, those below
$z=b$ go from one equilibrium in $\{(x,y,a):x^2+y^2=1-b^2\}$ to $\sigma_-$.

Next, choose $\delta>0$ small enough so that
$\rho'(z)\in [\rho'(a),\rho'(-\delta)]\subset (0,\infty)$ for $z\in [a,-\delta]$.
Then,
\begin{equation}
\label{weissman}
\rho'(z)(z^2-1)\leq \rho'(a)(a^2-1),\quad\quad\forall z\in [a,-\delta].
\end{equation}
Take $K=\{(x,y,z)\in M:z=-\delta\}$ the parallel circle at level $-\delta$ so $K$ is compact contained in $M^*$.

Divide $K$ into arcs of equal angular length $\frac{\pi}{2^{n-1}}$.
Let $E_n\subset K$ be the set of endpoints of these arcs. Then $E_n$ consists of $2^n$ evenly spaced points on $K$.

It follows that
\begin{equation}
\label{pelo1}
d(\varphi_t(p),\varphi_t(q))\geq \frac{1}{\sqrt{2}}(1-a^2)\frac{\pi}{2^{n-1}},\quad\quad\forall n\in\mathbb{N},\, p\neq q\in E_n,\,t\geq0,
\end{equation}
where $\varphi_t$ is the flow of $X$.

Given $p\in K$ we write the solution $\varphi_t(p)=(x(t),y(t),z(t))$ for $t\in\mathbb{R}$.
The function
$$
u(t)=\rho(z(t))\quad\quad\mbox{ for }\quad\quad t\geq0
$$
satisfies
$u'(t)=\rho'(z(t))z'(t)$ and, since
$$
z'(t)=(-x^2(t)-y^2(t))\rho(z(t))=(z^2(t)-1) \rho(z(t)),
$$
$u(t)$ satisfies the ODE
$$
u'(t)=\rho'(z(t))(z^2-1)u(t).
$$
So,
$$
\rho(z(t))=\rho(z(0))e^{\int_0^t\rho'(z(s))(z^2(s)-1)ds}.
$$
But
$a\leq z(s)\leq -\delta$ for $0\leq s\leq t$ and $t\geq0$ so
\eqref{weissman} yields
$$
\rho(z(t))\leq \rho(-\delta)e^{-\gamma t},\quad\quad\forall t\geq0,
$$
where $\gamma=\rho'(a)(1-a^2)>0$.

\begin{figure}
  \includegraphics[width=200pt]{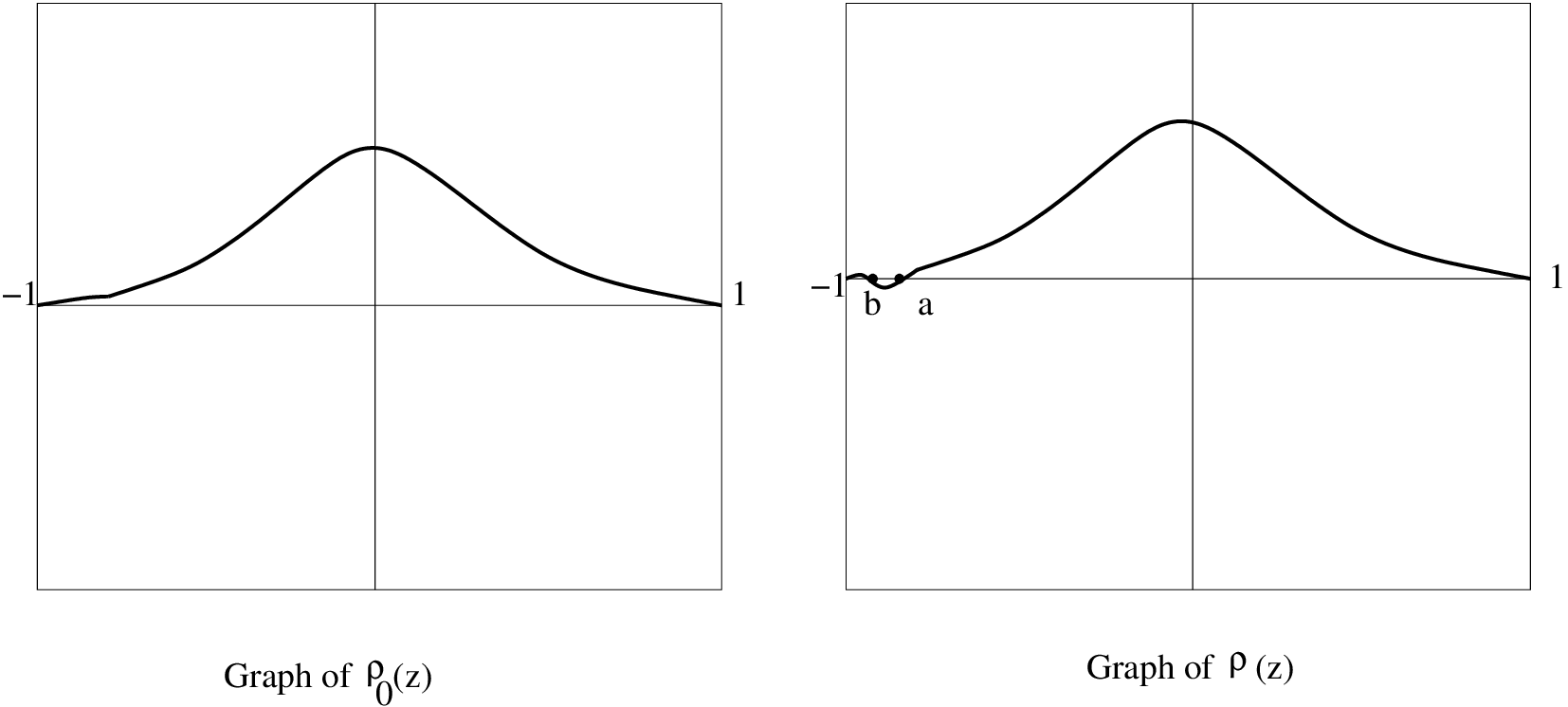}
 \caption{Graphs of $\rho_0$ and $\rho$}
  \label{fig2}
\end{figure}

Thus,
\begin{equation}
\label{pelo2}
\|X(\varphi_t(p))\|=\rho(z(t))\sqrt{1-z^2(t)}\leq \kappa e^{-\gamma t},\quad\quad\forall t\geq0,
\end{equation}
where $\kappa=\rho(-\delta)\sqrt{1-\delta^2}>0$.

Now take $t_n=\gamma^{-1}n$ for $n\in \mathbb{N}$.
We have from \eqref{pelo1} and \eqref{pelo2} that
$$
d^*_{t_n}(p,q)>\sqrt{2}\kappa^{-1}(1-a^2)\pi,\quad\quad\forall n\in\mathbb{N},\, p\neq q\in E_n.
$$
Consequently,
$$
\inf_{n\in\mathbb{N}}\inf_{p\neq p\in E_n}d^*_{t_n}(p,q)>0.
$$
On the other hand, $card(E_n)=2^n$ so
$$
\limsup_{n\to\infty}\frac{1}{t_n}\log card(E_n)=\gamma\log 2>0.
$$
Since $t_n\to\infty$, we obtain $e^*(X)>0$ from Proposition \ref{naro}.
\end{proof}

\section*{Funding}

\noindent
ER was partially supported by the European Union's European Research Council Marie Sklodowska-Curie grant No 101151716.
XW was partially supported by NSFC12071018 and the Fundamental Research Funds for the Central Universities.

\section*{Declaration of competing interest}
\label{Three}

\noindent
There is no competing interest.

\section*{Data availability}
\label{Four}

\noindent
No data was used for the research described in the article.


\begin{thebibliography}{10}






\bibitem{akm}
Adler, R.L., Konheim, A.G., McAndrew, M.H., 
Topological entropy,
{\em Trans. Amer. Math. Soc.} 114 (1965) 309--319.



\bibitem{acp}
 Arbieto, A., Cordeiro, W., Pacifico, M.J.,
 Continuum-wise expansivity and entropy for flows,
 {\em Ergodic Theory Dynam. Systems} 39 (2019), no. 5, 1190--1210.

\bibitem{a}
Artigue, A.,
Rescaled expansivity and separating flows,
{\em Disc. Cont. Dynam. Syst.} 38 (2018), 4433--4447.



\bibitem{acop}
Artigue, A., Cordeiro, W., Pacífico, M.J.,
N-expansive flows,
{\em Topology Appl.} 305 (2022), Paper No. 107869, 15 pp.












\bibitem{b}
Bowen, R.,
Entropy for group endomorphisms and homogeneous spaces,
{\em Trans. Amer. Math. Soc.} 153 (1971), 401--414.





\bibitem{bw}
Bowen, R., Walters, P.,
Expansive one-parameter flows.
{\em J. Differential Equations} 12 (1972), 180--193.







\bibitem{d}
Dinaburg, E.I.,
The relation between topological entropy and metric entropy,
{\em Dokl. Akad. Nauk SSSR} 190, (1970) 19--22 (Soviet Math. Dokl. 11 (1969), 13--16).








\bibitem{gw}
Guckenheimer, J., Williams, R.F.,
Structural stability of Lorenz attractors,
{\em, Inst. Hautes \'Etudes Sci. Publ. Math.} 50 (1979), 59--72.









\bibitem{hw}
Han, B., Wen, X.,
A shadowing lemma for quasi-hyperbolic strings of flows,
{\em J. Differential Equations} 264 (2018), 1--29.
















\bibitem{k}
Komuro, M.,
Expansive properties of Lorenz attractors,
{\em The theory of dynamical systems and its applications to nonlinear problems (Kyoto, 1984)}, 4–26, World Sci. Publishing, Singapore, 1984.











\bibitem{pdm}
Palis, J., de Melo, W.,
{\em Geometric theory of dynamical systems. An introduction},
Springer-Verlag, New York-Berlin, 1982.





\bibitem{r}
Rego, E.,
{\em Entropy Theory of Expansive Systems}
Thesis, Universidade Federal do Rio de Janeiro (2021).




\bibitem{rwy}
Rojas, A., Wen, X., Yang, Y.,
Sufficient conditions for rescaling expansivity,
{\em Ann. Fac. Sci. Toulouse Math.} (6) 33 (2024), no. 2, 447--467.



\bibitem{s}
Sun, W.,
Entropy of orthonormal $n$-frame flows,
{\em Nonlinearity} 14 (2001), 829--842.












\bibitem{u}
Ulcigrai, C.,
Slow chaos in surface flows,
{\em Boll. Unione Mat. Ital.} 14 (2021), no. 1, 231--255.





\bibitem{ww}
Wang, Y., Wen, X.,
A rescaled metric entropy of $C^1$-vector fields with singularities,
Preprint (2024).

\bibitem{wew}
Wen, X., Wen, L.,
A rescaled expansiveness of flows,
{\em Trans. Amer. Math. Soc.} 371 (2019), 3179--3207.


\bibitem{wy}
Wen, X., Yu, Y.,
Equivalent definitions of rescaled expansiveness,
{\em J. Korean Math. Soc.} 55 (2018), 593--604.


\bibitem{lg}
Young, L-S,
Entropy of continuous flows on compact $2$-manifolds,
{\em Topology} 16 (1977), 469--471.


\end{thebibliography}
\end{document}